\documentclass[12pt]{article}

\usepackage{amssymb,amsmath,fullpage,amsthm,amsfonts, verbatim, tikz, bbm}
\usepackage{url}
\usepackage{algorithm}
\usepackage{algpseudocode}
\usepackage{xcolor}
\definecolor{linkblue}{HTML}{003d73}
\definecolor{linkgreen}{HTML}{006161}
\definecolor{linkred}{HTML}{a11950}
\definecolor{csuorange}{HTML}{D9782D}
\definecolor{csugold}{HTML}{C8C372}
\definecolor{csublue}{HTML}{12A4B6}
\usepackage{hyperref}
\hypersetup{
	pdftitle={Distributions of Distances and Volumes of Balls in Homogeneous Lens Spaces},
	pdfauthor={Brenden Balch, Chris Peterson, and Clayton Shonkwiler},
	pdfsubject={differential geometry},
	pdfkeywords={lens spaces, random geometry, data models},
	colorlinks=true,
	linkcolor=linkblue,
	citecolor=linkgreen,
	urlcolor=linkred
}
\usepackage{authblk}
\usepackage{ytableau}
\usepackage[nameinlink]{cleveref}
\usepackage{bigints}
\usepackage{booktabs}
\usepackage{rotating}

\usetikzlibrary{arrows}
\usetikzlibrary{positioning}
\usetikzlibrary{cd}

\newtheorem{theorem}{Theorem}
\newtheorem*{theorem*}{Theorem}

\newtheorem{cor}[theorem]{Corollary}
\newtheorem*{cor*}{Corollary}

\newtheorem*{moments}{Theorem~\ref*{thm:moments}}
\newtheorem*{asymptotics}{Corollary~\ref*{cor:asymptotic moments}}
\newtheorem*{mgf}{Theorem~\ref*{thm:mgf}}
\newtheorem*{ballvol}{Theorem~\ref*{thm:ball volume}}

\theoremstyle{definition}

\newcommand{\R}{\mathbb{R}}

\newcommand{\C}{\mathbb{C}}
\newcommand{\N}{\mathbb{N}}

\newcommand{\Sph}{\mathbb{S}}
\newcommand{\vol}{\operatorname{Vol}}

\newcommand{\J}{\mathbf{j}}

\newcommand{\Fl}{F\ell}

\newcommand{\dVol}{\operatorname{dVol}}

\newcommand{\du}{\,\mathrm{d}u}

\newcommand{\dtheta}{\,\mathrm{d}\theta}
\newcommand{\deta}{\,\mathrm{d}\eta}

\newcommand*\pFqskip{8mu}
\catcode`,\active
\newcommand*\pFq{\begingroup
        \catcode`\,\active
        \def ,{\mskip\pFqskip\relax}%
        \dopFq
}
\catcode`\,12
\def\dopFq#1#2#3#4#5{%
        {}_{#1}\nobreak\hspace{-.05em}F_{#2}\biggl[\genfrac..{0pt}{}{#3}{#4};#5\biggr]%
        \endgroup
}

\title{Distributions of Distances and Volumes of Balls in Homogeneous Lens Spaces}
\author{Brenden Balch}
\author{Chris Peterson}
\author{Clayton Shonkwiler}
\affil{Department of Mathematics, Colorado State University, Fort Collins, CO}
\date{}

\begin{document}
	
\maketitle

\begin{abstract}
Lens spaces are a family of manifolds that have been a source of many interesting phenomena in topology and differential geometry. Their concrete construction, as quotients of odd-dimensional spheres by a free linear action of a finite cyclic group, allows a deeper analysis of their structure.
In this paper, we consider the {\it problem of moments} for the distance function between randomly selected pairs of points on homogeneous three-dimensional lens spaces. We give a derivation of a recursion relation for the moments, a formula for the $k^{\text{th}}$ moment, and a formula for the moment generating  function, as well as an explicit formula for the volume of balls of all radii in these lens spaces.	
\end{abstract}

\

\section{Introduction} Given a set of data, what is the best guess for the random process that produced the data? Attempts to answer special cases of this question have motivated new developments in statistics, mathematics, and machine learning. As a starting point, one would like to understand whether the observed data has a distribution differing from what is ``expected.'' However, determining what is expected can be quite subtle when the data takes values on a manifold, though when the manifold is homogeneous, there are additional tools that one can use to simplify the problem. At an intuitive level, a homogeneous manifold is a space in which each point is indistinguishable from any other point.

For distance data, one would ideally like to check whether the distribution of pairwise distances is compatible with the corresponding distribution on the manifold. In a previous paper~\cite{BPS}, we considered the problem of computing the expected distances between randomly drawn points on manifolds of partially oriented flags. These manifolds generalize projective spaces and other Grassmannians and form a large family of homogeneous spaces. The examples in which we had the most success computing expected distances turn out to be (scaled) lens spaces; that is, quotients of an odd-dimensional sphere by the free action of a cyclic group. In this paper we go beyond simple expectations and determine precisely the distributions of distances between pairs of random points in all homogeneous three-dimensional lens spaces.

These distributions are examples of \emph{distance distributions} (or sometimes \emph{shape distributions} or \emph{distance histograms}), which make sense on arbitrary metric measure spaces, and are often used for geometric classification and shape analysis~\cite{Berrendero:2016di,Bonetti:2005bm,Boutin:2004gb,Brinkman:2012ef,Memoli:2011jg,Memoli:2018wk,Osada:2002tq}. Our results provide a strong statistical baseline against which to compare data on lens spaces, which have recently been applied to data science~\cite{Polanco:2019ug}, appear frequently in the cosmography literature~\cite{AFS12, Aurich:2012by, TOM19, Uzan:2004dr}, and are the natural setting for spherical data with cyclic symmetries.

To establish notation, each pair of positive integers $(n,m)$ with $n>m$ and ${\gcd(n,m)=1}$ determines a three-dimensional lens space $L(n;m)$ which is a quotient of the 3-sphere $\Sph^3$ by the cyclic group of order $n$. By requiring the quotient to be a Riemannian submersion, we induce a Riemannian metric on $L(n;m)$, which turns out to be homogeneous when $m=1$ or $n-1$. Moreover, $L(n;1)$ and $L(n;n-1)$ are isometric, so to understand distance distributions on homogeneous lens spaces it suffices to consider those lens spaces of the form $L(n;1)$.

As a first step, we determine all moments of distance (i.e., expected values of powers of distance) by solving a recurrence relation that they satisfy:
\begin{theorem}\label{thm:moments} 
	For each $k \geq 0$ and each $n \geq 2$, the $k$th moment of distance on $L(n;1)$ is
	\begin{multline*}
		I_{n,k} = \frac{1}{(k+1)(k+2)}\!\left[ \frac{4}{k+3}\! \left(\frac{\pi}{n}\right)^{\!\!k+2}\!\!\! \pFq{1}{2}{1}{\frac{k+4}{2}, \frac{k+5}{2}}{-\frac{\pi^2}{n^2}} \right. \\
		\left. + \tan\frac{\pi}{n} \left(\! n\! \left(\frac{\pi}{2}\right)^{\!\!k+1} \!\!\!\pFq{1}{2}{1}{\frac{k+3}{2},\frac{k+4}{2}}{-\frac{\pi^2}{4}} - 2\!\left(\frac{\pi}{n}\right)^{\!\!k+1} \!\!\!\pFq{1}{2}{1}{\frac{k+3}{2},\frac{k+4}{2}}{-\frac{\pi^2}{n^2}}\right)\!\right], 
	\end{multline*}
	where for $n=2$ this is interpreted as the limit of the above expression as $n \to 2$, and ${}_{1}\nobreak\hspace{-.05em}F_{2}$ is a hypergeometric function whose definition we recall on \cpageref{hypergeometeric def} below.
\end{theorem}

The alternating finite sum formula given in~\eqref{eq:first order solution} is typically more useful for small $k$, but one virtue of this formulation in terms of hypergeometric functions is that it is easy to extract asymptotic information:

\begin{cor}\label{cor:asymptotic moments}
	As $k \to \infty$ the $k$th moment of distance grows like
$$I_{2,k} \sim \frac{2}{k} \!\left(\frac{\pi}{2}\right)^{\!\!k} \ \   {\rm and\  }  \ \ 	I_{n,k} \sim \frac{n}{k^2}\!\left(\frac{\pi}{2}\right)^{\!\!k+1} \!\!\!\!\tan \frac{\pi}{n}\ \ \ {\rm for} \ n\geq 3. $$
\end{cor}

A more attractive and systematic packaging of the moments is in the form of the moment-generating function of distance:

\begin{theorem}\label{thm:mgf}
	The moment-generating function of distance on $L(n;1)$ is
	\[
		M_n(t) = \begin{cases} \frac{4}{\pi(4+t^2)}\! \left(\!\frac{2(e^{t \pi/2}\!-1)}{t} + te^{t\pi/2}\! \right)  & \text{if } n=2\\
		\frac{2n}{\pi(4+t^2)}\! \left(\! \frac{2(e^{t \pi/n}\!-1)}{t} + \tan \frac{\pi}{n} \left(e^{t\pi/2}\!-e^{t \pi/n}\right)\!\!\right) & \text{if } n\geq 3.\end{cases}
	\]
\end{theorem}

We then use the moment-generating function to determine the cumulative distribution function of distance, which (up to scaling) simply reports volumes of balls. Consequently, our probabilistic approach to studying distances on lens spaces yields the following purely geometric result:

\begin{theorem}\label{thm:ball volume}
	For $n \geq 2$, the volume of a ball of radius $r$ in $L(n;1)$ is
	\[
		V_n(r) = \begin{cases} 2\pi(r-\sin r \cos r) & \text{if } r \leq \frac{\pi}{n} \\ \frac{2\pi^2}{n} - 2\pi \cos^2 r \tan\frac{\pi}{n} & \text{else.}\end{cases}
	\]
\end{theorem}

Notice, in particular, that this formula for volume extends beyond the injectivity radius $\frac{\pi}{n}$ of $L(n;1)$, in contrast to most results about volumes of balls in Riemannian manifolds (e.g.,~\cite{Gray:1979ih}). In addition to the potential applications of these ideas to data problems, this seems to be a novel result to add to existing knowledge about the geometry and topology of lens spaces~\cite{ALEX, brody, Ikeda:1979wb,Lehoucq:2003js, PY03, reidemeister, SL05, Tanaka:1979ue, Viana:2018ho}.

We describe our perspective, provide basic background material on lens spaces, and give the setting in which algorithms and analytic computations are to be made in~\Cref{LENS}. In \Cref{EXP} we describe algorithms for sampling random points and determining their distance apart. In addition, we present the results of several Monte Carlo experiments that illustrate differences between distributions of distances on homogeneous and non-homogeneous lens spaces. \Cref{THEO} contains the main theoretical results of the paper.

\section{Lens Spaces}\label{LENS}
Three-dimensional lens spaces are a family of manifolds that arise as the orbit space of a finite cyclic group acting freely on the unit 3-sphere. More precisely, let \(Z_n = \{e^{i2\pi k /n} \in \C ~|~ 1 \leq k \leq n\}\) denote the cyclic group of order \(n\) and consider \(\Sph^3 = \{(\alpha,\beta) \in \C^2~|~|\alpha|^2 + |\beta|^2 = 1\} \). Given $n,m \in \mathbb N$ with \(\gcd(m,n)=1\), there is a free action of \(Z_n\) on \(\Sph^3\) defined by 
\[
\omega \cdot (\alpha,\beta) = (\omega \alpha, \omega^m \beta),
\] 
for each $\omega \in Z_n$. The resulting orbit space is the lens space \(L(n;m)\).

To visualize $L(n;m)$, we can look at the fundamental domain of the $Z_n$ action on $\Sph^3 \subseteq \C^2$, as in \Cref{fig:lens space}. The fundamental domain of the rotation $e^{2\pi i/n}$ in the first factor is an arc of length $\frac{2\pi}{n}$ in the unit circle in the $z_1$-plane of $\C^2$. All points in $\Sph^3$ with first coordinate in such a fundamental domain form a lens-shaped domain as pictured. The top and bottom faces of the lens consist of all points lying on geodesics connecting an endpoint of the arc to all points in the unit circle in the $z_2$-plane: these are hemispheres of unit 2-spheres meeting at an angle of $\frac{2\pi}{n}$ along the unit circle in the $z_2$-plane. Since the endpoints of the arc are identified under the $\frac{2\pi}{n}$ rotation in the $z_1$-coordinate, the bottom face is identified with the top face by this rotation. However, this identification happens with a $\frac{2\pi m}{n}$ twist in the $z_2$-coordinate, so that the green sector in the bottom face is glued to the green sector in the top face (in the picture, $m=2$). 


\begin{figure}[t]
	\centering
		\includegraphics[height=2in]{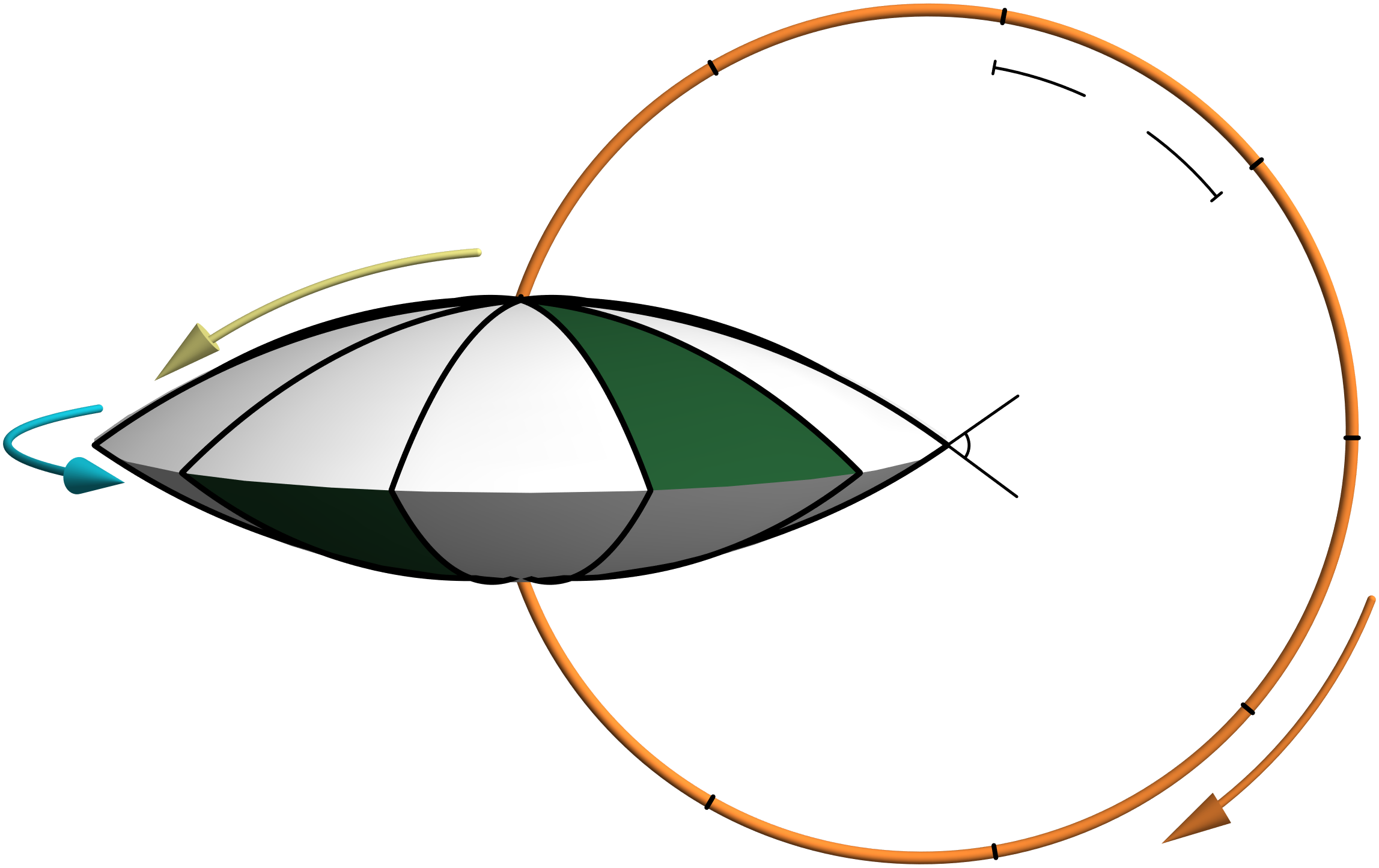}
		\put(-12,12){\textcolor{csuorange}{$\theta_1$}}
		\put(-240,69){\textcolor{csublue}{$\theta_2$}}
		\put(-182,101){\textcolor{csugold}{$\eta$}}
		\put(-49,127){\begin{rotate}{-32}\tiny $\frac{2\pi}{n}$\end{rotate}}
		\put(-66,68.8){\tiny $\frac{2\pi}{n}$}
	\caption{A fundamental domain of the lens space $L(n;m)$. The arrows indicate the directions of the join coordinates $(\theta_1,\theta_2,\eta)$ that will be defined on \cpageref{join coords}.}
	\label{fig:lens space}
\end{figure}
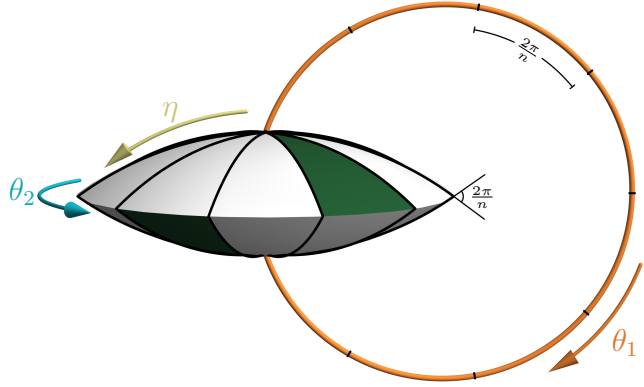

Lens spaces were introduced by Tietze~\cite{Tietze} and have historically provided interesting examples of manifolds which cannot be distinguished by homology or homotopy groups. For example, \(L(5;1)\) and \(L(5;2)\) are not homeomorphic (nor even homotopy equivalent) despite the fact that \(\pi_1(L(5;1)) \cong \pi_1(L(5;2))\) and \(H_\bullet(L(5;1)) \cong H_\bullet(L(5;2))\)~\cite{ALEX}. In fact, the lens spaces \(L(n;m_1)\) and \(L(n;m_2)\) are homotopy equivalent if and only if  \(m_1m_2 = \pm a^2\pmod n\) for some \(a \in \N\), and are homeomorphic if and only if \(m_1 = \pm m_2^{\pm1}\pmod n\)~\cite{brody,reidemeister}. Using these criteria, one can easily conclude that \(L(7;1)\) and \(L(7;2)\) are examples of manifolds which are homotopy equivalent but not homeomorphic.

In addition to their topological structure, lens spaces have geometric structure. The round metric on $\Sph^3$ induces a unique metric on \(\Sph^3/Z_n = L(n;m)\) that makes \(\pi: \Sph^3 \to L(n;m)\) a Riemannian submersion. A result of Ikeda and Yamamoto~\cite{Ikeda:1979wb} implies that two three-dimensional lens spaces are isometric if and only if they are homeomorphic. This result, combined with work of Tanaka~\cite{Tanaka:1979ue}, shows that the spectrum of the Laplacian uniquely determines a three-dimensional lens space among all Riemannian manifolds. An explicit orthonormal eigenbasis for the Laplacian is given in~\cite{Lehoucq:2003js}. Moreover, the isoperimetric problem has been solved in all lens spaces $L(n;m)$ with $n$ large enough~\cite{Viana:2018ho}.

With respect to this Riemannian metric, some lens spaces are homogeneous, meaning the isometry group acts transitively. Theorem 7.6.6 of Wolf~\cite{Wolf:2011tw} says that \(\Sph^d/G\) is homogeneous if and only if the group \(G\) has a \emph{Clifford representation}\,---\,that is, a faithful orthogonal representation \(\rho:G \to O(d+1)\) such that \(\rho(g) = \pm I\) or half of the eigenvalues of \(\rho(g)\) are \(\lambda \in \Sph^1\) and the other half are \(\bar{\lambda}\). The action of \(Z_n\) on \(\Sph^3\) has the faithful orthogonal representation \(\rho:Z_n \to O(4)\) given by 
\[
\rho(\omega) = 
\begin{pmatrix}
\cos 2\pi/n& -\sin 2\pi/n& 0&0\\
\sin 2\pi/n & \cos 2\pi/n &0&0\\
0&0&\cos 2\pi m/n& -\sin 2\pi m/n\\
0&0&\sin 2\pi m/n & \cos 2\pi m/n\\
\end{pmatrix},
\]
and has eigenvalues \(e^{i2\pi / n}, e^{-i2\pi / n}, e^{i2\pi m/ n}\) and \(e^{-i2\pi m/ n}\). Hence \(L(n;m)\) is homogeneous precisely when \(m = 1\) or \(m = n-1\). Since \(L(n;1)\) and \(L(n;n-1)\) are homeomorphic and hence isometric, we may simply take \(m=1\) when dealing with homogeneous lens spaces.

\subsection{Coordinate systems}
Using the natural group structure on \(\Sph^3\) given by its identification with the unit quaternions, we can describe an isomorphism between \(\Sph^3\) and \(SU(2)\). Writing quaternions in the form $\alpha + \beta \J$ for $\alpha, \beta \in \C$, define \(\varphi: \Sph^3 \to SU(2)\) by
\[
\varphi : \alpha + \beta \J \mapsto 
\begin{pmatrix}
\alpha & -\beta\\
\overline{\beta} & \overline{\alpha}
\end{pmatrix},
\]
where \(\overline{\zeta}\) denotes the complex conjugate of \(\zeta\). It is easy to check that \(\varphi\) is a Lie group isomorphism. The action of \(Z_n\) on \(\Sph^3\) then induces an action on \(SU(2)\) given explicitly by
\begin{align} \label{act}
\omega \cdot \begin{pmatrix}
\alpha & -\beta\\
\overline{\beta} & \overline{\alpha}
\end{pmatrix}
= 
\begin{pmatrix}
\omega \alpha & -\omega^m\beta\\
\overline{\omega^m\beta} & \overline{\omega \alpha}
\end{pmatrix}.
\end{align}
Describing the lens space in this way will make our computations straightforward. The idea is that we can easily generate random elements of \(SU(2)\) according to Haar measure (which corresponds to the uniform probability measure on $\Sph^3$), compute the orbits explicitly, and then distances between orbits correspond to distances in the lens space.

\label{join coords} For homogeneous lens spaces, we will be able to make explicit analytic calculations in \Cref{THEO}. To do so, we'll parametrize \(\Sph^3\) using \emph{join coordinates}, which realize the 3-sphere as the join of two circles. Since \(\Sph^3 = \{(\alpha,\beta) \in \C^2 \ \arrowvert \ |\alpha|^2 + |\beta|^2 = 1\}\), we can write \(\alpha = e^{i \theta_1} \cos\eta\) and \(\beta = e^{i \theta_2} \sin\eta\) for \(\theta_1,\theta_2 \in [-\pi,\pi)\) and \(\eta \in [0,\pi/2]\). This can also be expressed in Cartesian coordinates on \(\R^4\) as
\begin{align}
x &= \cos \theta_1 \cos \eta \nonumber \\
y &= \sin \theta_1 \cos \eta \label{eq:join coords}\\
z &= \cos \theta_2 \sin \eta \nonumber \\
w &= \sin \theta_2 \sin \eta. \nonumber
\end{align}

These coordinates easily yield the volume form \(\dVol_{\Sph^3} = \cos \eta \sin \eta \deta \wedge \dtheta_1 \wedge \dtheta_2\), and the volume form induced by the Riemannian submersion metric on the homogeneous lens space \(L(n;1)\) is \({\dVol_{L(n;1)}  = \cos \eta \sin \eta \deta \wedge \dtheta_1 \wedge \dtheta_2}\), where now \(\theta_1, \theta_2 \in [-\pi/n,\pi/n)\). A straightforward calculation shows that \(\vol(L(n;1)) = 2\pi^2/n^2\).

\section{Algorithms and Experiments} \label{EXP}
In this section we'll provide an algorithm for a Monte Carlo experiment. We then use this as a guide for analysis on higher moments. 

Our aim is to describe a Monte Carlo simulation which will allow us to approximate expected (Riemannian) distances between two points in \(L(n;m)\). We will use \Cref{alg:RandSU} to randomly generate elements of \(SU(n)\).
\begin{algorithm}
  \caption{Random Special Unitary Matrix}\label{alg:RandSU}
  \begin{algorithmic}[1]
    \Function{RandSU}{$n$}
    \State \(A,B \gets \text{random $n\times n$ Gaussian}\)
    \State \(C \gets A+iB\) \Comment{where \(i = \sqrt{-1}\)}
    \State \(Q \gets \textproc{GramSchmidt}{(C)}\)
    \State \(Q_{1,n} \gets \frac{1}{\det(Q)}Q_{1,n}\) \Comment{\(Q_{1,n}\) is the last column of \(Q\)}
    \EndFunction
  \end{algorithmic}
\end{algorithm}
\hfill


We will use the Riemannian distance function on \(SU(2)\) (see \cite{grass}), then use the Riemannian submersion \(\pi: SU(2) \to L(n;m)\) to obtain a distance function on the lens space. Suppose that \(A,B \in SU(2)\), and let \(\lambda_1,\lambda_2\) be the eigenvalues of \(AB^*\).  For a nonzero complex number $z = x + yi$,  we let $\log z$ denote the principal value logarithm whose imaginary part lies in the interval $(-\pi,\pi]$. We have
\[
d(A,B) = \frac{1}{\sqrt{2}}\sqrt{|\log \lambda_1|^2 + |\log \lambda_2|^2}
\]
or, since $\lambda_2 = \overline{\lambda}_1$, $d(A,B) = |\log \lambda_1|$. To compute distances on \(L(n;m)\), we first compute the orbits, then compute pairwise distances between the elements of each orbit, and finally take the minimum of all distances computed. Thus for \([A], [B] \in L(n;m)\), we have
\[
d([A],[B]) = \min_{1\leq j,k \leq n} \{ d(\omega^j\cdot A, \omega^k \cdot B)\},
\]
where $\omega = e^{2\pi i/n}$. This leads to \Cref{alg:distgen}.
\begin{algorithm}
\caption{Expected Distance on $L(n;m)$}\label{alg:distgen}
\begin{algorithmic}[1]
\State \(D \gets [0]*N\) \Comment{Begin with a list of \(N\) zeroes}
\For{$k\gets 1 , N$}
   \State \(A \gets \textproc{RandSU}(2)\)
   \State \(B \gets \textproc{RandSU}(2)\)
   \State \(\text{orbitdata} \gets [0]*n \times n\) \Comment{Initialize $n\times n$ zero matrix}
   \For{$ i \gets 1, n$}
   	\For{$ j \gets 1,n$}
   	\State \( \text{orbitdata}(i,j) \gets d(\omega^i \cdot A,\omega^j \cdot B)\) 
	\EndFor
   \EndFor
   \State \(D(k) \gets \textproc{min}(\text{orbitdata})\)
\EndFor
\end{algorithmic}
\Return \(\textproc{mean}(D)\)
\end{algorithm}

For example, using \Cref{alg:distgen} with \(N=1,\!000,\!000\), we estimate the expected distances between random points on \(L(5;1)\) and \(L(5;2)\) to be approximately \(0.85897\) and \(0.80378\), respectively, reflecting the fact that these lens spaces are not isometric (nor even homeomorphic); see \Cref{fig:L51and52}. The corresponding estimates for \(L(7;1)\) and \(L(7;2)\) are \(0.82641\) and \(0.73641\), respectively, again reflecting the fact that these spaces are neither isometric nor homeomorphic, though they are homotopy equivalent.

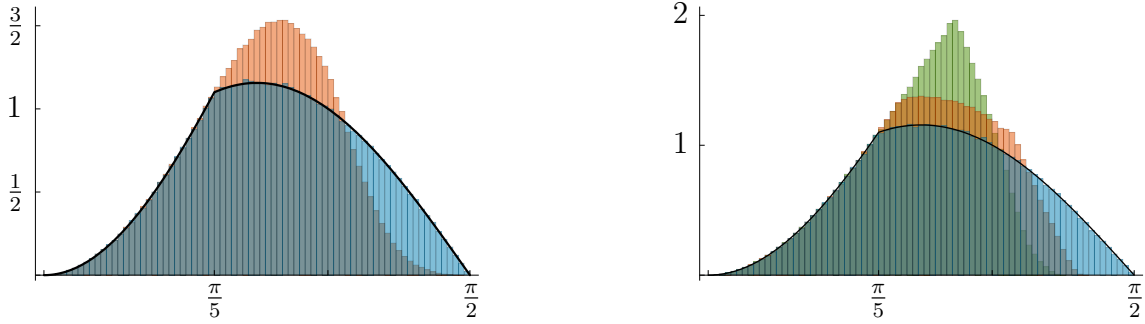
\begin{figure}[t]
	\centering
		\includegraphics[height=1.5in]{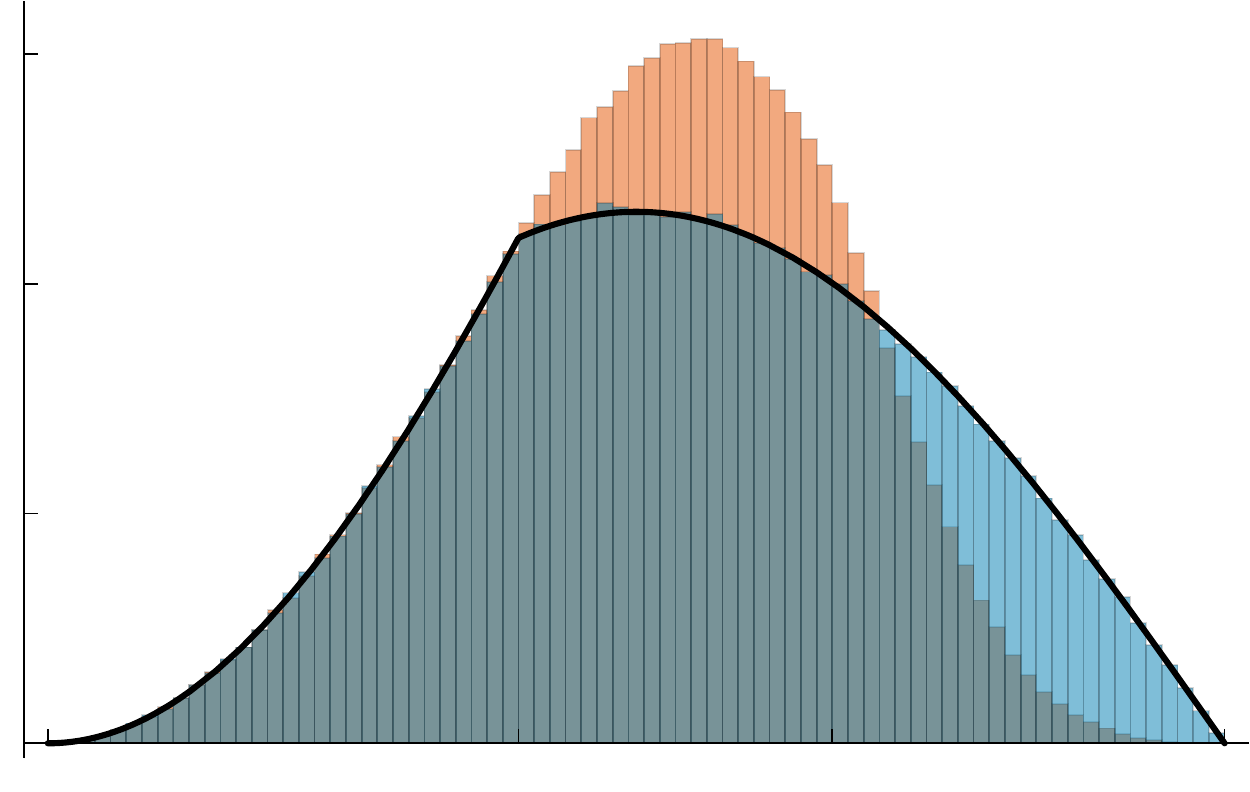}
		\hspace{1in}
		\includegraphics[height=1.5in]{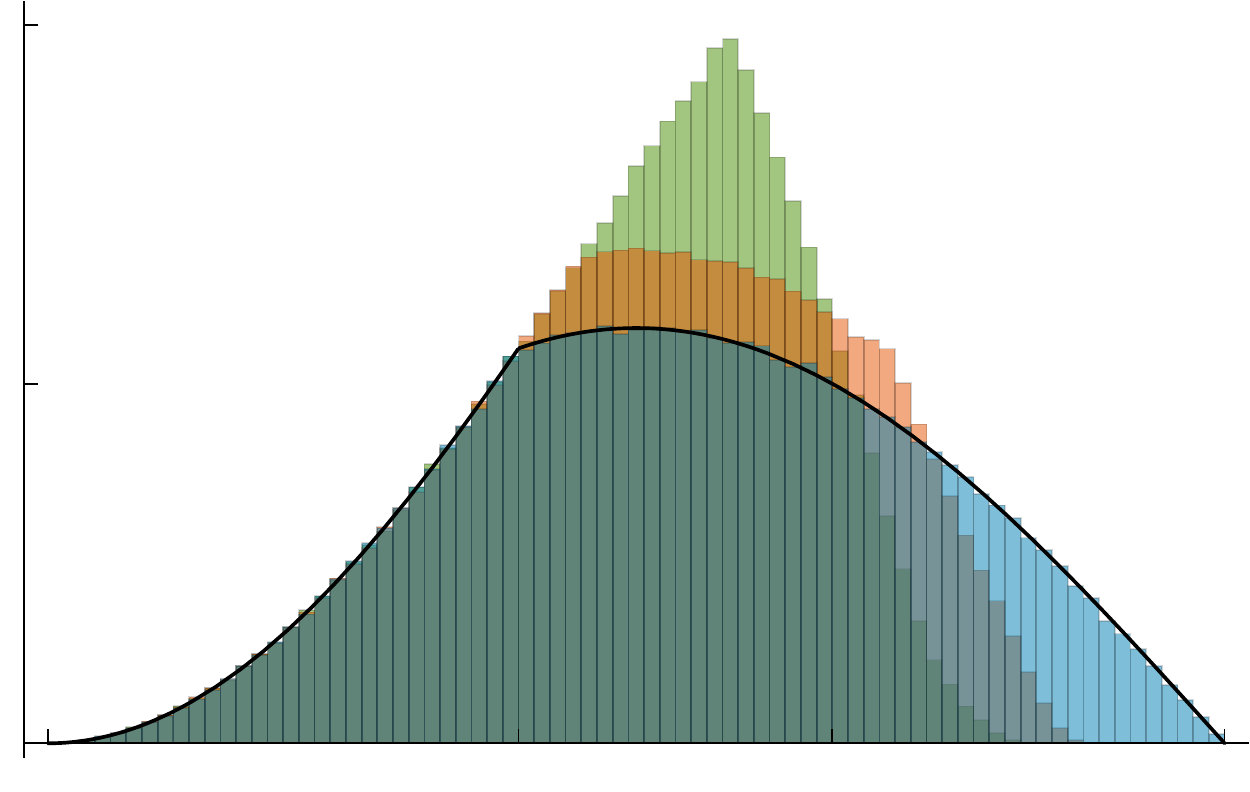}
		\put(-430,35.05){$\frac{1}{2}$}
		\put(-430,66.56){$1$}
		\put(-430,98.04){$\frac{3}{2}$}
		\put(-178,52.8){$1$}
		\put(-178,101.95){$2$}
		\put(-104,-5){$\frac{\pi}{5}$}
		\put(-7,-5){$\frac{\pi}{2}$}
		\put(-355,-5){$\frac{\pi}{5}$}
		\put(-258,-5){$\frac{\pi}{2}$}
	\caption[Histograms of distances in $L(5;1)$ and $L(5;2)$]{On the left are histograms of distances between 1,000,000 random pairs of points in $L(5;1)$ [blue] and $L(5;2)$ [red], computed using \Cref{alg:distgen}; the curve shows the true density of distances in $L(5;1)$ from~\eqref{eq:distancePDF}. The right shows histograms of distances in $L(5;2)$ from 1,000,000 random points to different fixed points, where the fixed points are the images in $L(5;2)$ of $SU(2)$ elements of the form $\begin{pmatrix} \cos \phi & -\sin \phi \\ \sin \phi & \cos \phi \end{pmatrix}$, where $\phi = 0$ [blue], $\pi/8$ [red], and $\pi/4$ [green]; the curve shows the density of distances from random points in $L(5;1)$ to any fixed point, again from~\eqref{eq:distancePDF}. In particular, whereas the distributions of distances from random points to any fixed point in the homogeneous space $L(5;1)$ are all the same, these distributions vary with the fixed point in the non-homogeneous $L(5;2)$.}
	\label{fig:L51and52}
\end{figure}

In the homogeneous case, it actually suffices to fix a representative of a fixed orbit, then check the distances between the chosen representative and each element of the other orbit. To see this, note that 
\begin{align*}
 d([A],[B]) &= \min_{1\leq i,j \leq n} \{ d(\omega^i\cdot A, \omega^j \cdot B)\\
 &=  \min_{1\leq k \leq n} \{ d(\omega^k \cdot AB^*, I).
\end{align*}
If $A, B$ are chosen according to Haar measure on $SU(2)$, then $AB^*$ will also be distributed according to Haar measure, which is definitionally invariant under the (left or right) action of $SU(2)$ on itself. Hence, when doing a computational experiment, we can generate one random element of \(SU(2)\), compute the orbit under the action, and then compute the distances from each element in the orbit to the identity. This yields the less computationally expensive \Cref{alg:dist}.

\begin{algorithm}
\caption{Expected Distance on $L(n;1)$}\label{alg:dist}
\begin{algorithmic}[1]
\State \(D \gets [0]*N\) 
\For{$l\gets 1 , N$}
   \State \(A \gets \textproc{RandSU}(2)\)
   \State \(\text{orbitdata} \gets [0]*n\) 
   \For{$ k \gets 1, n$}
   \State \( \text{orbitdata}(k) \gets d(\omega^k \cdot A,I)\) 
   \EndFor
   \State \(D(l) \gets \textproc{min}(\text{orbitdata})\)
\EndFor
\end{algorithmic}
\Return \(\textproc{mean}(D)\)
\end{algorithm}

For \(N=1,\!000,\!000\), a na\"ive Matlab implementation of \Cref{alg:distgen} gives the estimate \({E[d; L(5,1)] \approx 0.85897}\) in about 940 seconds on a laptop, whereas \Cref{alg:dist} yields \(E[d; L(5,1)] \approx  0.85921\) in about 86 seconds.

\section{Distributions of Distances}\label{THEO}
We now restrict to the case that \(L(n;m)\) is homogeneous; as previously mentioned, we can (and will) assume in what follows that \(m=1\). In this section we derive an analytic description of the distributions of distances on all the $L(n;1)$ lens spaces.

As a first step to understanding these distributions of distances, we will compute the \(k\)th moment of distance between 2 random points in \(L(n;1)\). We now work in join coordinates~\eqref{eq:join coords}, and we think of points in $L(n;1)$ as orbits of points in $\Sph^3$. Since $L(n;1)$ is homogeneous, we may fix one point  to be (the orbit of) the point $q = (1,0,0,0)$. The fundamental domain of the $Z_n$ action centered at this point (depicted in \Cref{fig:lens space}) is determined by the join coordinate inequalities
\[
	-\frac{\pi}{n} \leq \theta_1, \theta_2 < \frac{\pi}{n},
\]
so computing the expectation of $k$th power of distance in $L(n;1)$ is equivalent to computing the expectation of $[d_{\Sph^3}(p,q)]^k$, where $p$ varies over this fundamental domain. With $p$ written in join coordinates, $d_{\Sph^3}(p,q) = \arccos(p \cdot q) = \arccos(\cos \theta_1 \cos \eta)$, so the $k$th moment of distance is exactly
\begin{align*}
I_{n,k} := \mathbb{E}[d^k; L(n;1)] &= \frac{1}{\vol(L(n;1))}\int_{L(n;1)}\!\!\!\!\!\! [d_{L(n;1)}([p],[q])]^k ~\dVol_{L(n;1)}\\
&= \frac{n^2}{2\pi^2}\int_{-\pi/n}^{\pi/n} \int_{-\pi/n}^{\pi/n} \int_0^{\pi/2} \!\!\!\!\!\!\arccos^k(\cos \theta_1 \cos \eta)  \cos \eta \sin \eta \deta \dtheta_1 \dtheta_2.
\end{align*}

Obviously,
\begin{equation}\label{eq:zeroth moment}
	I_{n,0} = 1.
\end{equation}

For \(k \geq 1\), the integral expression for \(I_{n,k}\) can be simplified somewhat by integrating out \(\theta_2\), observing that the integrand is even in \(\theta_1\), and making the substitution \(\cos u = \cos \theta_1 \cos \eta\). Doing so produces the integral
\begin{equation}\label{eq:integral}
	I_{n,k} = \frac{2n}{\pi}\int_0^{\pi/n} \!\!\!\!\!\!\sec^2 \theta_1 \int_{\theta_1}^{\pi/2} \!\!\! u^k \cos u \sin u \du \dtheta_1.
\end{equation}

Notice that this integral is improper for $n=2$. For $n\geq 3$ we can apply the reduction formula~\cite[2.631.1]{GR} for the inner integral to compute the first moment
\begin{equation}\label{eq:first moment}
	I_{n,1}  = \frac{\pi}{2n} + \frac{n-2}{4}\tan\frac{\pi}{n}
\end{equation}
and the relation
\begin{equation}\label{eq:recurrence}
	I_{n,k} =- \frac{k(k-1)}{4} I_{n,k-2} + \frac{1}{k+1} \left(\frac{\pi}{n}\right)^k + \frac{n}{2\pi} \left[ \left(\frac{\pi}{2}\right)^k\!\!- \left(\frac{\pi}{n}\right)^k \right] \tan \frac{\pi}{n}
\end{equation}
for \(k \geq 2\).

We can solve this recurrence using standard methods. The following theorem expresses the solution in terms of {\it generalized hypergeometric functions} $\pFq{p}{q}{a_1, a_2, \dots , a_p}{b_1, b_2, \dots , b_q}{z}$. In the definition of this class of functions, it  is convenient  to  introduce the  {\it Pochhammer symbol} $(a)_n$,  defined by the rule 
\[
		(a)_n = \begin{cases} 1 & \text{if } n=0 \\ a(a+1) \cdots (a+n-1) & \text{if } n\geq 1. \end{cases}
	\]
Equivalently, so long as $a$ is not a nonpositive integer $(a)_n = \frac{\Gamma(a+n)}{\Gamma(a)}$.
	
In terms of  the  Pochhammer symbol, the generalized hypergeometric function is defined by the series \label{hypergeometeric def}  $$\pFq{p}{q}{a_1, a_2, \dots , a_p}{b_1, b_2, \dots , b_q}{z} = \sum_{n=0}^{\infty}\frac{(a_1)_n (a_2)_n \cdots (a_p)_n}{(b_1)_n (b_2)_n \cdots (b_q)_n} \frac{z^n}{n!},$$ provided none of the $b_1, \dots , b_q$ is a nonpositive integer. When $p \leq q$, the series converges for all $z$ and ${}_{p}\nobreak\hspace{-.05em}F_{q}$ is entire.

\begin{moments}
	For each $k \geq 0$ and each $n \geq 3$, the $k$th moment of distance on $L(n;1)$ is
	\begin{multline}\label{eq:moments} 
		I_{n,k} = \frac{1}{(k+1)(k+2)}\!\left[ \frac{4}{k+3}\! \left(\frac{\pi}{n}\right)^{\!\!k+2}\!\!\! \pFq{1}{2}{1}{\frac{k+4}{2}, \frac{k+5}{2}}{-\frac{\pi^2}{n^2}} \right. \\
		\left. + \tan\frac{\pi}{n} \left(\! n\! \left(\frac{\pi}{2}\right)^{\!\!k+1} \!\!\!\pFq{1}{2}{1}{\frac{k+3}{2},\frac{k+4}{2}}{-\frac{\pi^2}{4}} - 2\!\left(\frac{\pi}{n}\right)^{\!\!k+1} \!\!\!\pFq{1}{2}{1}{\frac{k+3}{2},\frac{k+4}{2}}{-\frac{\pi^2}{n^2}}\right)\!\right]. 
	\end{multline}
	Values for small $k$ are given in \Cref{table:moments}.
\end{moments}

\begin{table}
	\begin{center}
		\def\arraystretch{1.2}
		\begin{tabular}{ll}
			$k$ & $I_{n,k}$ \\
			\midrule
	 		$0$ & $1$ \\ 
	 		$1$ & $\frac{\pi }{2 n}+\left(\frac{1}{4}n -\frac{1}{2}\right)\! \tan \frac{\pi }{n}$ \\ 
	 		$2$ & $-\frac{1}{2}+\frac{\pi ^2}{3 n^2}+\left(\frac{\pi  }{8}n-\frac{\pi }{2 n}\right)\!\tan\frac{\pi }{n}$ \\ 
	 		$3$ & $-\frac{3 \pi }{4 n}+\frac{\pi ^3}{4 n^3}+\left(\!\frac{\pi ^2 -6}{16}n+\frac{3}{4}-\frac{\pi ^2}{2 n^2}\!\right)\!\tan\frac{\pi }{n}$ \\ 
	 	   	$4$ & $\frac{3}{2}-\frac{\pi ^2}{n^2}+\frac{\pi ^4}{5 n^4}+\left(\!\frac{\pi ^3 -12\pi}{32}n+\frac{3 \pi }{2 n}-\frac{\pi ^3}{2 n^3}\!\right)\!\tan\frac{\pi }{n}$ \\ 
	 	   	$5$ & $\frac{15 \pi }{4 n}-\frac{5 \pi ^3}{4 n^3}+\frac{\pi ^5}{6 n^5}+\left(\!\frac{\pi ^4 -20\pi^2+120}{64}n-\frac{15}{4}+\frac{5 \pi ^2}{2 n^2}-\frac{\pi ^4}{2 n^4}\!\right)\!\tan\frac{\pi }{n}$ \\ 
	 	   	$6$ & $-\frac{45}{4}+\frac{15 \pi ^2}{2 n^2}-\frac{3 \pi ^4}{2 n^4}+\frac{\pi ^6}{7 n^6}+\left(\!\frac{\pi ^5 - 30 \pi^3 + 360\pi }{128}n-\frac{45 \pi }{4 n}+\frac{15 \pi ^3}{4 n^3}-\frac{\pi ^5}{2 n^5}\!\right)\! \tan \frac{\pi }{n}$ \\
	 	   	$7$ & $-\frac{315 \pi }{8 n}+\frac{105 \pi ^3}{8 n^3}-\frac{7 \pi ^5}{4 n^5}+\frac{\pi ^7}{8 n^7}+\left(\!\frac{\pi ^6-42 \pi ^4+840 \pi ^2-5040}{256}n+\frac{315}{8}-\frac{105 \pi ^2}{4 n^2}+\frac{21 \pi ^4}{4 n^4}-\frac{\pi ^6}{2 n^6}\!\right)\! \tan \frac{\pi }{n}$
		\end{tabular}
	\end{center}
	\caption{Values of the $k$th moment of distance $I_{n,k}$ for small $k$ and $n \geq 3$.}
	\label{table:moments}
\end{table}

\begin{proof}
	While the difference equation~\eqref{eq:recurrence} is second-order, the even and odd $I_{n,k}$ are independent of each other, so we can separately reduce each to a first-order difference equation and then solve that first-order equation. 
	
	For example, if $k=2m$ is even, then defining $y_m := I_{n,2m}$ and index-shifting allows us to re-write~\eqref{eq:recurrence} as
	\[
		y_{m+1} =- \frac{(2m+2)(2m+1)}{4} y_m + \frac{1}{2m+3} \left(\frac{\pi}{n}\right)^{\!\!2m+2}\!\! + \frac{n}{2\pi} \left[ \left(\frac{\pi}{2}\right)^{\!\!2m+2} \!\!\!- \left(\frac{\pi}{n}\right)^{\!\!2m+2}\right]\! \tan \frac{\pi}{n}
	\]
	with initial condition $y_0 = I_{n,0} =1$ from~\eqref{eq:zeroth moment}.
	
	This is in the standard form $y_{m+1} = g_m y_m + h_m$ for general first-order linear difference equations, and hence has solution
	\begin{align}\label{eq:first order solution}
		y_m & = \prod_{j=0}^{m-1}g_j \left(y_0 + \sum_{j=0}^{m-1} \frac{h_j}{\prod_{\ell=0}^j g_\ell} \right)  \\
		& = (-1)^m \frac{(2m)!}{2^{2m}}\! \left(\!1 + \sum_{j=0}^{m-1} \frac{(-1)^{j+1}}{(2j+3)!}\!\left(\!\frac{2\pi}{n}\!\right)^{\!\!2j+2} \!\!\!+ \frac{n}{2\pi} \tan \frac{\pi}{n} \sum_{j=0}^{m-1} \frac{(-1)^{j+1}}{(2j+2)!} \!\left(\!\! \pi^{2j+2} - \left(\!\frac{2\pi}{n}\!\right)^{\!\!2j+2}\right) \!\!\right) \nonumber
	\end{align}
	after some simplification. 
	
	In turn, each of the finite sums becomes one of the hypergeometric functions in~\eqref{eq:moments}. For example,
	\begin{align*}
		1 + \sum_{j=0}^{m-1} \frac{(-1)^{j+1}}{(2j+3)!}\!\left(\!\frac{2\pi}{n}\!\right)^{\!\!2j+2} & = \frac{n}{2\pi}\sum_{j=0}^m \frac{(-1)^{j}}{(2j+1)!}\!\left(\!\frac{2\pi}{n}\!\right)^{\!\!2j+1} \\
		& = \frac{n}{2\pi} \sum_{j=0}^\infty \frac{(-1)^{j}}{(2j+1)!}\!\left(\!\frac{2\pi}{n}\!\right)^{\!\!2j+1} \!\!\!\!- \frac{n}{2\pi} \sum_{j=m+1}^\infty \frac{(-1)^{j}}{(2j+1)!}\!\left(\!\frac{2\pi}{n}\!\right)^{\!\!2j+1} \\
		& = \frac{n}{2\pi} \sin \frac{2\pi}{n} - \frac{(-1)^{m+1}}{(2m+3)!}\!\left(\!\frac{2\pi}{n}\!\right)^{\!\!2m+2} \sum_{i=0}^\infty \frac{1}{(m+2)_i(m+\frac{5}{2})_i}\!\left(\!\!-\frac{\pi^2}{n^2}\right)^{\!\!i}.
	\end{align*}
	After multiplying each term by $1 = \frac{i!}{i!} = \frac{(1)_i}{i!}$, the remaining sum is the standard power series representation of $\pFq{1}{2}{1}{m+2, m+\frac{5}{2}}{-\frac{\pi^2}{n^2}}$.
	
	Simplifying the remaining terms in~\eqref{eq:first order solution} and replacing $2m$ with $k$ yields the solution~\eqref{eq:moments} for the even moments.
	
	On the other hand, notice that we can solve the difference equation~\eqref{eq:recurrence} for $I_{n,1}$ independent of the value of $I_{n,-1}$. Therefore, if we define $z_m = I_{n,2m-1}$ for $m \geq 1$, we can choose the initial condition $z_0$ arbitrarily. If we choose $z_0 = 1$,\footnote{We emphasize that $I_{n,-1} \neq 1$; in fact, it is not too hard to show that
	\[
		I_{n,-1} = \frac{n}{\pi} \left[\gamma - \operatorname{Ci}\!\left(\!\frac{2\pi}{n}\!\right) + \log\!\left(\!\frac{2\pi}{n}\!\right) + \left(\!\operatorname{Si}(\pi) - \operatorname{Si}\!\left(\!\frac{2\pi}{n}\!\right)\!\!\right)\tan \frac{\pi}{n} \right],
	\] where $\gamma \approx 0.577$ is the Euler--Mascheroni constant and $\operatorname{Ci}$ and $\operatorname{Si}$ are the cosine integral and sine integral functions, respectively.} then the difference equation and initial condition for $z_m$ are essentially identical to those in the problem we just solved. Indeed, solving the system and plugging in $k=2m-1$ at the end yields the exact same expression~\eqref{eq:moments} for the odd moments, completing the proof.
\end{proof}

We can't plug $n=2$ into the expressions~\eqref{eq:first moment} and~\eqref{eq:recurrence}, but taking the limit as $n \to 2$ gives the corresponding values of the improper integral~\eqref{eq:integral}:
\begin{align*}
	I_{2,1} & = \frac{1}{\pi} + \frac{\pi}{4} \\
	I_{2,k} & = - \frac{k(k-1)}{4} I_{2,k-2} + \frac{1}{k+1} \!\left(\frac{\pi}{2}\right)^{\!\!k} \!+ \frac{k}{\pi}\!\left(\frac{\pi}{2}\right)^{\!\!k-1}.
\end{align*}

The solution of this initial value problem (together with $I_{2,0}=1$) is simply the limit of~\eqref{eq:moments} as $n \to 2$:

\begin{cor}\label{cor:n=2 moments}
	The $k$th moment of distance on $L(2;1) = \mathbb{RP}^3$ is
	\begin{multline*}
		I_{2,k} = \frac{1}{k+1}\!\left(\frac{\pi}{2}\right)^{\!\!k}\! \left[ 2\, \pFq{1}{2}{1}{\frac{k+3}{2},\frac{k+4}{2}}{-\frac{\pi^2}{4}} \right. \\
		\left. +\frac{\pi^2}{(k+2)(k+3)}\!\left(\! \pFq{1}{2}{1}{\frac{k+4}{2},\frac{k+5}{2}}{-\frac{\pi^2}{4}} - \frac{4}{k+4}\pFq{1}{2}{2}{\frac{k+5}{2},\frac{k+6}{2}}{-\frac{\pi^2}{4}}\right)\!\right].
	\end{multline*}
\end{cor}

We point out that the partially oriented flag manifolds $\Fl((1,1,1);\{\{1\},\{2\},\{3\}\})$ and $\Fl((1,1,1),\{\{1\},\{2,3\}\})$ considered in our previous paper~\cite{BPS} are (up to a global scale factor of 2) the lens spaces $L(2;1)$ and $L(4;1)$, respectively, and indeed the expected values of distance that we computed on those spaces were exactly $2I_{2,1} = \frac{2}{\pi} + \frac{\pi}{2}$ and $2I_{4,1} = 1 + \frac{\pi}{4}$.

For small $k$ the finite sum formula~\eqref{eq:first order solution} is typically more useful than~\eqref{eq:moments}\,---\,and, indeed, the finite sum is what we see in \Cref{table:moments}\,---\,but one virtue of \Cref{thm:moments} and \Cref{cor:n=2 moments} is that we can easily determine the asymptotic behavior of $I_{n,k}$ as $k \to \infty$ by retaining only the leading terms in the power series representations of the hypergeometric functions.

\begin{asymptotics}
	For fixed $n\geq 3$, the asymptotic growth of the $k$th moment of distance on $L(n;1)$ as $k \to \infty$ is
	\[
		I_{n,k} \sim \frac{n}{k^2}\!\left(\frac{\pi}{2}\right)^{\!\!k+1} \!\!\!\!\tan \frac{\pi}{n}.
	\]
	For $n=2$, the asymptotic growth of the $k$th moment of distance on $L(2;1) = \mathbb{RP}^3$ is
	\[
		I_{2,k} \sim \frac{2}{k} \!\left(\frac{\pi}{2}\right)^{\!\!k}.
	\]
\end{asymptotics}

On the other hand, if we fix $k$ and let $n$ get large, only the middle term in~\eqref{eq:moments} survives:

\begin{cor}\label{cor:large n limit}
	For fixed $k \geq 0$, 
	\[
		\lim_{n \to \infty} I_{n,k} = \frac{\pi}{(k+2)(k+1)} \!\left(\frac{\pi}{2}\right)^{\!\!k+1}\!\!\! \pFq{1}{2}{1}{\frac{k+3}{2}, \frac{k+4}{2}}{-\frac{\pi^2}{4}}.
	\]
	Values for small $k$ are given in \Cref{table:large n limit}.
\end{cor}

\begin{table}
	\begin{center}
	\def\arraystretch{1.2}
	\begin{tabular}{ll}
		$k$ & $\displaystyle\lim_{n \to \infty} I_{n,k}$\\
		\midrule
	 	$0$ & $1$ \\ 
	 	$1$ & $\frac{1}{4}\pi$ \\ 
		$2$ & $-\frac{1}{2}+\frac{\pi  }{8}\pi$ \\ 
	 	$3$ & $\frac{\pi ^2 -6}{16}\pi$ \\ 
	 	$4$ & $\frac{3}{2}+\frac{\pi ^3 -12\pi}{32}\pi$ \\ 
	 	$5$ & $\frac{\pi ^4 -20\pi^2+120}{64}\pi$ \\ 
	 	$6$ & $-\frac{45}{4}+\frac{\pi ^5 - 30 \pi^3 + 360\pi }{128}\pi$ \\
	 	$7$ & $\frac{\pi ^6-42 \pi ^4+840 \pi ^2-5040}{256}\pi$
	\end{tabular}
	\end{center}
	\caption{$\displaystyle \lim_{n \to \infty}I_{n,k}$ for small $k$. The coefficient of $\pi$ is the coefficient of $n \tan \frac{\pi}{n}$ in the corresponding entry in \Cref{table:moments}, and the remaining term is the constant term from \Cref{table:moments}.}
	\label{table:large n limit}
\end{table}

Another way to package the information contained in \Cref{thm:moments} is by computing the moment-generating function of distance:

\begin{mgf}
	For $n \geq 3$, the moment-generating function of distance on $L(n;1)$ is
	\begin{equation}\label{eq:mgf}
		M_n(t) = \frac{2n}{\pi(4+t^2)}\! \left(\! \frac{2(e^{t \pi/n}\!-1)}{t} + \tan \frac{\pi}{n} \left(e^{t\pi/2}\!-e^{t \pi/n}\right)\!\!\right).
	\end{equation}
	For $n=2$, the moment-generating function is
	\[
		M_2(t) = \frac{4}{\pi(4+t^2)}\! \left(\!\frac{2(e^{t \pi/2}\!-1)}{t} + te^{t\pi/2}\! \right).
	\]
\end{mgf}

\begin{proof}
	By definition,
	\begin{equation}\label{eq:mgf integral}
		M_n(t) = \mathbb{E}(e^{t d}; L(n;1)) = \frac{2n}{\pi}\! \int_0^{\pi/n} \!\!\!\!\!\! \sec^2 \theta_1 \int_{\theta_1}^{\pi/2} \!\! e^{t u} \cos u \sin u \du \dtheta_1
	\end{equation}
	using the same substitution that produced~\eqref{eq:integral}. Using the identity $\sin 2u = 2 \sin u \cos u$ and integrating by parts twice yields 
	\[
		M_n(t) = \frac{2n}{\pi(4+t^2)}\! \left(\! \frac{e^{t\pi/n}\!-1}{t} + \tan\! \frac{\pi}{n}\, e^{t\pi/2} \!-\! \int_0^{\pi/n}\!\!\! e^{t \theta_1} (t\tan \theta_1  + \sec^2 \theta_1 -1) \dtheta_1\!\right)
	\]
	for $n \geq 3$. Integrating the first term inside the integral by parts produces a term which cancels the second, and the rest is straightforward.
	
	For $n=2$, evaluating the indefinite integral~\eqref{eq:mgf integral} boils down to taking the limit of~\eqref{eq:mgf} as $n \to 2$, which produces the desired expression for $M_2(t)$.
\end{proof}

We can recover the probability density function (pdf) $f_n$ of distance as the inverse Laplace transform of $M_n(-t)$:
\begin{align}
	f_2(x) & = \frac{4}{\pi} \sin^2 x \nonumber \\
	f_n(x) & = \frac{2n}{\pi} \!\left(\sin^2 x + \Theta\!\left(x-\pi/n\right)\!\left(-\sin^2 x + \sin x \cos x \tan \frac{\pi}{n}\right)\! \right), \label{eq:distancePDF}
\end{align}
where $\Theta$ is the Heaviside function which is zero for negative values and 1 for positive values. See \Cref{fig:density}.

\begin{figure}[t]
	\centering
		\includegraphics[height=1.5in]{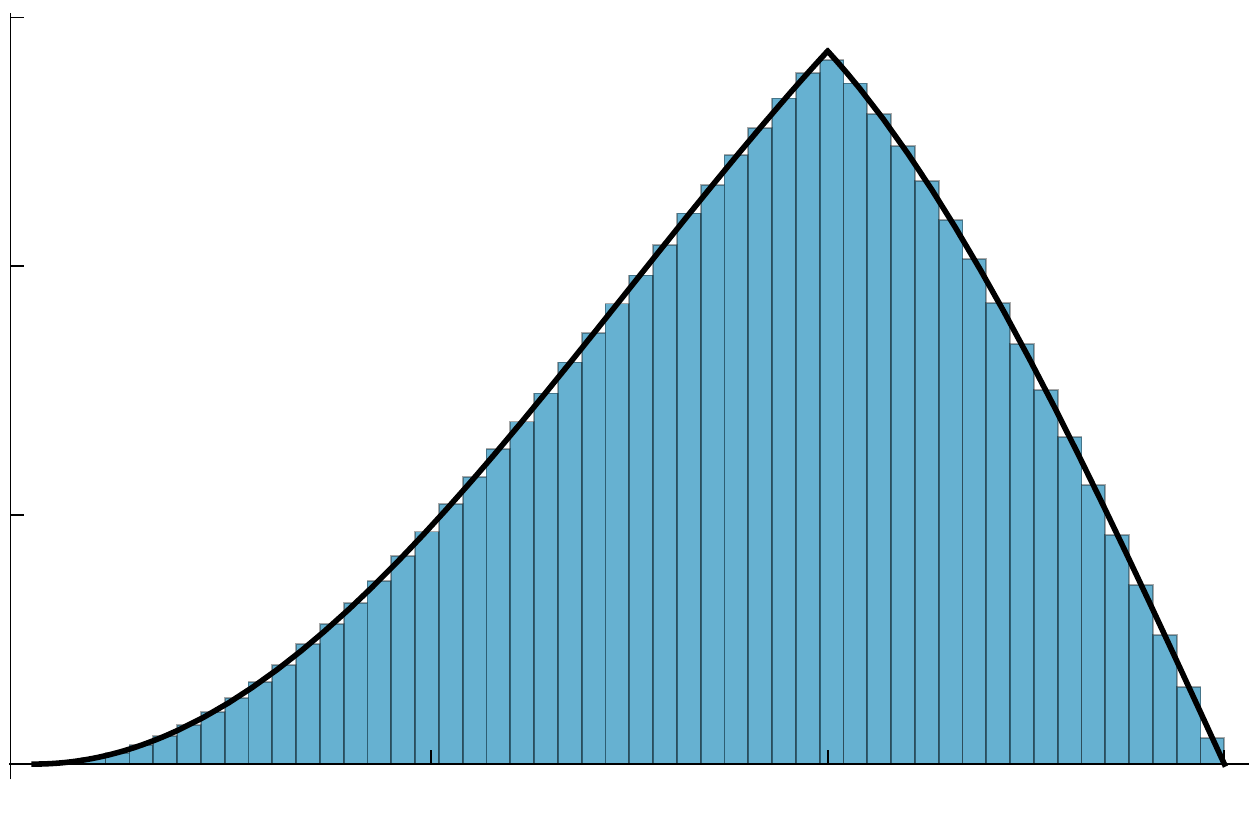}
		\put(-113.5,-5){$\frac{\pi}{6}$}
		\put(-60.5,-5){$\frac{\pi}{3}$}
		\put(-7,-5){$\frac{\pi}{2}$}
		\put(-175,36.4){$\frac{1}{2}$}
		\put(-175,69.6){$1$}
		\put(-175,103.3){$\frac{3}{2}$}
	\caption{Histogram of distances between 10,000,000 random points on $L(3;1)$ generated by \Cref{alg:dist} compared to the pdf ${f_3(x)=\frac{6}{\pi} \left(\sin^2 x + \Theta(x-\pi/3)\!\left(-\sin^2 x + \sqrt{3}\sin x \cos x\right)\!\right)}$.}
	\label{fig:density}
\end{figure}
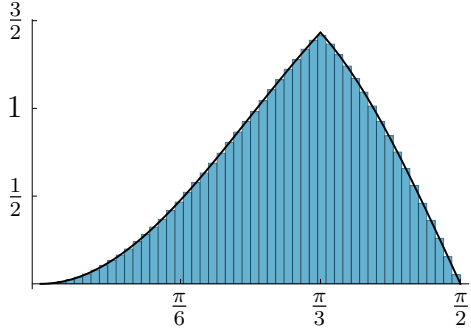

As $n \to \infty$ we see that $f_n(x) \to \sin 2x$, the pdf of the \emph{sine distribution} introduced by Gilbert in the study of moon craters~\cite{Edwards:2000wk,Gilbert:1893tb}. It is not so surprising to see this distribution: as $n \to \infty$ the lens spaces $L(n;1)$ converge in the Gromov--Hausdorff sense to a 2-sphere of radius $1/2$, and the distance distribution on this sphere is exactly the sine distribution.

In turn, given the pdf, we can integrate to get the cumulative distribution function $F_n(x)$ of distance on $L(n;1)$:
\begin{align*}
	F_2(x) & = \frac{2}{\pi}(x - \sin x \cos x) \\
	F_n(x) & = \frac{n}{\pi}\!\left(x - \sin x \cos x + \Theta\!\left(x-\pi/n\right)\!\left(\frac{\pi}{n}-x+\sin x \cos x - \cos^2 x \tan\frac{\pi}{n} \right)\!\right)\!;
\end{align*}
see \Cref{fig:cdfs}.

\begin{figure}[t]
	\centering
		\includegraphics[height=1.5in]{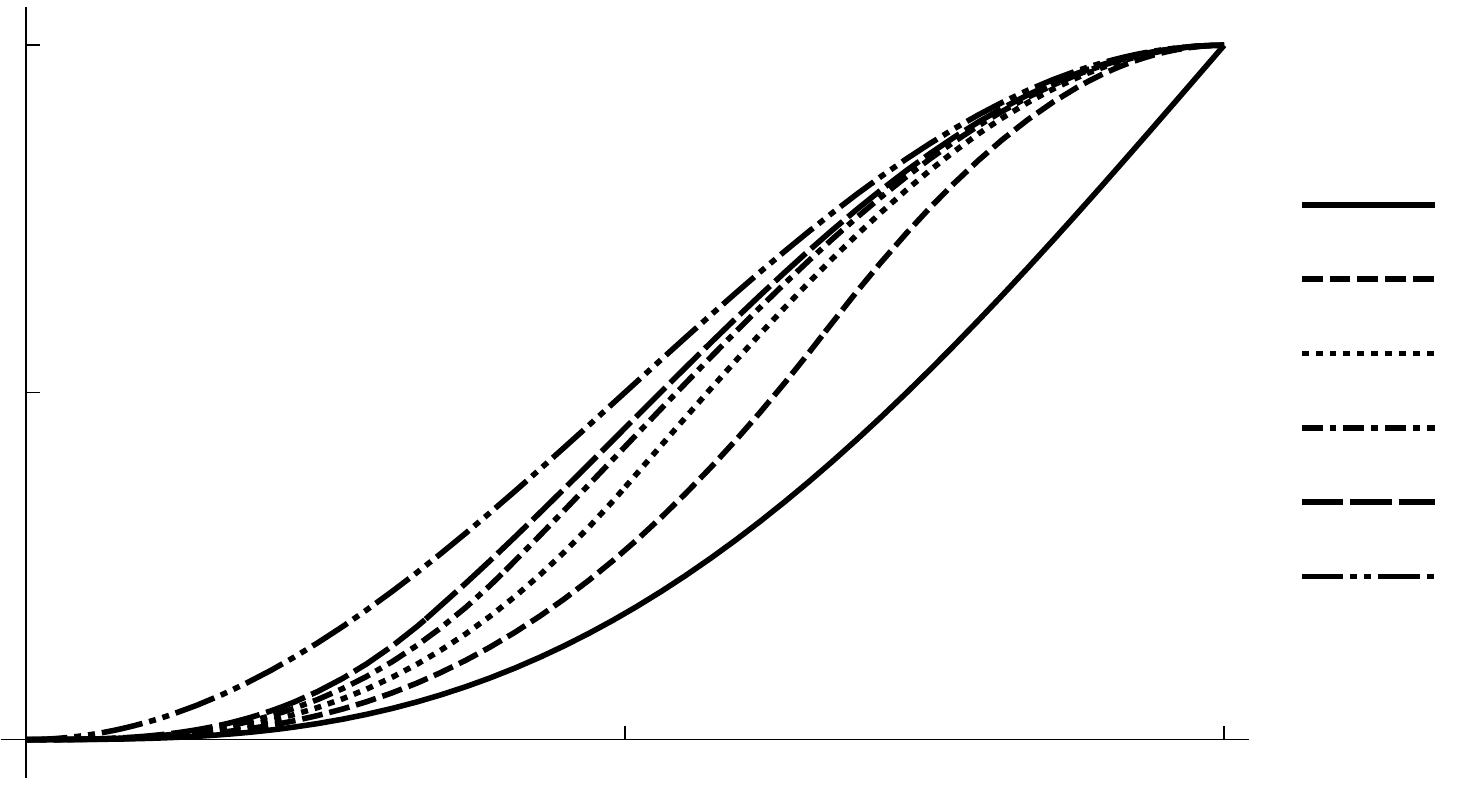}
		\put(-4,78.5){\tiny $n=2$}
		\put(-4,68.28){\tiny $n=3$}
		\put(-4,58.06){\tiny $n=4$}
		\put(-4,47.84){\tiny $n=5$}
		\put(-4,37.62){\tiny $n=6$}
		\put(-4,27.4){\tiny $n\to \infty$}
		\put(-209.5,51.3){$\frac{1}{2}$}
		\put(-209.5,99){$1$}
		\put(-122,-5){$\frac{\pi}{4}$}
		\put(-39.5,-5){$\frac{\pi}{2}$}
	\caption{The cumulative distribution function of distance on $L(n;1)$ for $2\leq n \leq 6$ and in the limit as $n \to \infty$.}
	\label{fig:cdfs}
\end{figure}
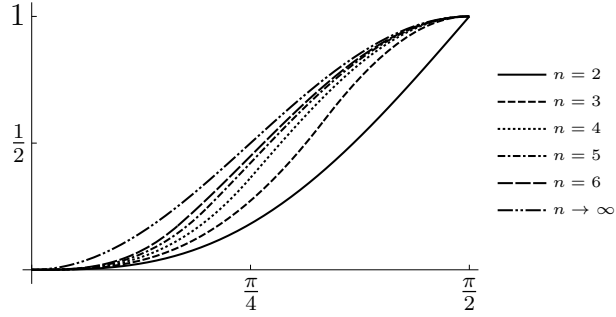

By definition,
\[
	F_n(x) = \mathbb{P}(d(p,q) \leq x) = \frac{\vol B_q(x)}{\vol L(n;1)}
\]
where $q \in L(n;1)$ is any fixed point and $p \in L(n;1)$ is random; since $L(n;1)$ is homogeneous this is independent of $q$. Hence, we can compute the volume $V_n(r) := \vol B_q(r)$ of a ball of radius $r$ in $L(n;1)$ as 
\[
	V_n(r) = \vol (L(n;1)\!)F_n(r) = \frac{2\pi^2}{n} F_n(r).
\]
This proves:
\begin{ballvol}
	For $n \geq 2$, the volume of a ball of radius $r$ in $L(n;1)$ is
	\[
		V_n(r) = \begin{cases} 2\pi(r-\sin r \cos r) & \text{if } r \leq \frac{\pi}{n} \\ \frac{2\pi^2}{n} - 2\pi \cos^2 r \tan\frac{\pi}{n} & \text{else.}\end{cases}
	\]
\end{ballvol}

Since the diameter of $L(n;1)$ is $\frac{\pi}{2}$, $V_n$ is only defined on $[0,\frac{\pi}{2}]$, so we never reach the second case when $n=2$. Also, $2\pi(r-\sin r \cos r)$ is simply the volume of a ball of radius $r$ in $\Sph^3$; not surprisingly, things get interesting only when $r>\frac{\pi}{n}$, the injectivity radius of $L(n;1)$.

Thinking in these geometric terms, the pdfs from~\eqref{eq:distancePDF} are scaled areas of spheres. Rescaling by the same $\frac{2\pi^2}{n}$ factor as above yields the surface area $A_n(r)$ of the sphere of radius $r$ centered at any point in $L(n;1)$:
\[
	A_n(r) = \begin{cases} 4\pi \sin^2 r & \text{if } r \leq \frac{\pi}{n} \\ 4\pi \sin r \cos r \tan\frac{\pi}{n} & \text{else} . \end{cases}
\]

\section{Concluding Remarks}\label{FIN}
Three-dimensional lens spaces are a family of topological/geometric objects that have played a historical role in the development of manifold theory. Their interest derives both from their ease of construction and as examples of manifolds exhibiting unusual phenomena. They appear across several disciplines including topology, geometry, cosmography, and data science, and are a natural setting for spherical data with cyclic symmetries. While lens spaces have been well studied from varying perspectives, we are unaware of other sources which consider distance distributions on them.

Distance distributions have been used in geometric classification and can be used to understand general metric measure spaces. While they can often be approximated effectively using Monte Carlo techniques,  it would be interesting to determine analytic expressions for distance distributions on a broader class of manifolds. For manifolds which are not homogeneous spaces, the distribution of distances from a fixed point depends on the point. In other words, the volume formula for a ball is dependent on the location of the center of the ball in the manifold. In turn, integrating the distribution of distances from a fixed point as the fixed point varies over the manifold yields the distribution of distances between pairs of random points.

Non-homogeneous lens spaces, both in three and in higher dimensions, are particularly tractable examples of non-homogeneous manifolds, so in these spaces it may be feasible to find analytic expressions for the distributions of distances both from a fixed point and between random points.

\section*{Acknowledgments}

We thank all the participants in the Pattern Analysis Lab at Colorado State University for their energy, ideas, and ongoing inspiration, Tom Needham for helpful conversations, and the National Science Foundation (CCF--BSF:CIF \#1712788, ATD \#1830676, CP) and the Simons Foundation (\#354225, CS) for their support.

\bibliographystyle{plain}
\bibliography{bib}
\end{document}